\documentclass[12pt]{amsart}

\usepackage{amssymb}
\usepackage[all]{xy}
\usepackage{graphicx}
\usepackage[usenames]{color}

\newtheorem{theorem}{Theorem}[section]
\newtheorem{proposition}[theorem]{Proposition}
\newtheorem{lemma}[theorem]{Lemma}
\newtheorem{corollary}[theorem]{Corollary}
\newtheorem{definition}[theorem]{Definition}
\newtheorem{remark}[theorem]{Remark}

\numberwithin{equation}{section}

\usepackage{color}
\begin{document}

\baselineskip=15.5pt

\title[Moduli of unstables]{Moduli of unstable bundles of HN-length two with fixed algebra of endomorphisms}

\author{L. Brambila-Paz}

\address{CIMAT, Apdo. Postal 402,C.P. 36240, Guanajuato, Mexico}

\email{lebp@cimat.mx}

\author{Roc\'io R\'ios Sierra}

\address{CIMAT, Apdo. Postal 402,C.P. 36240, Guanajuato, Mexico}

\email{rocio.rios@cimat.mx}


\date{\today}

\thanks{Both authors acknowledges the support of CONACYT, in particular of CONACYT grant 251938.}

\begin{abstract}
Let $X$ be a smooth irreducible complex projective curve of genus $g \geq  2$ and $U_{{\mu _1}}(n,d)$ the moduli scheme of indecomposable vector bundles over $X$ with fixed Harder-Narasimhan type $\sigma=(\mu _1, \mu_2)$.  In this paper, we give necessary and sufficient conditions for a vector bundle $E\in U_{{\mu _1}}(n,d)$  to have $\mathbb{C}[x_1,\dots , x_k]/(x_1,\dots , x_k)^2$ as its algebra of endomorphisms. Fixing the dimension of the algebra of endomorphisms we obtain a stratification of $ U_ {\mu_1} (n, d) $ such that each stratum $ U_ {\mu_1} (n, d, k) $ 
 is an algebraic variety, moreover, a coarse moduli space.  A particular case of interest is when the unstable bundles are simple. In that case the moduli space is fine. Topological properties of $ U_ {\mu_1} (n, d, k) $ will depend on the generality of the curve $X$. Such results differ from the corresponding results for the moduli space of stable bundles, where non-emptiness, dimension etc. are independent of the curve.

\end{abstract}

\maketitle

\section{Introduction}\label{intro}

Let $X$ be a complex smooth projective variety. It is well known that the moduli space of semistable bundles with fixed invariants exists, and that sometimes is fine, i.e. there exists a universal family. Some moduli spaces of unstable bundles have been constructed by adding some extra information: the moduli spaces of  unstable vector bundles of rank $2$ and $3$  were constructed in \cite{bm} and in \cite{rocio}, respectively, when $\dim X=1$ and the algebra of endomorphisms is fixed; in \cite{str} when $X$ is the projective plane and in \cite{banc1} when $X=\mathbb{P}^3(\mathbb{C})$, in both cases they consider the degree of instability.  J. M. Drezet studied in \cite{drezet} the case of very unstable bundles when $\dim X >2$.  In \cite{hk}  the authors construct the moduli spaces of pure sheaves with fixed Harder-Narasimhan type which have some additional data called a $m$-rigidification. The moduli spaces of unstable sheaves via non-reductive GIT has been studied by V. Hoskins, G. Berczi, J. Jackson and F. Kirwan in \cite{frances}.

Unless otherwise stated we assume now that $X$ is a smooth irreducible complex projective curve of genus $g \geq  2.$ The aim of this paper is to study unstable bundles over $X$ using their algebra of endomorphisms.
The advantage of using algebra of endomorphisms lies in the fact that it allow us to construct coarse moduli spaces, and even fine moduli spaces.

In order to state our results let us recall that any vector bundle over $X$  has a unique filtration (called Harder-Narasimhan filtration)
$$ 0 = E_0\subset E_1 \subset \cdots \subset E_m = E$$
such that for $0 \leq i \leq  m-1, \  E_{i+1}/E_{i}$ is semistable and
\begin{equation}\label{eq01}
\mu (E_1) > \mu (E_2/E_1) > \cdots > \mu (E_m/E_{m-1}).
\end{equation}
The sequence of slopes $\sigma=(\mu (E_1) ,\mu (E_2/E_1), \cdots , \mu (E_m/E_{m-1}))$ is called the {\it Harder-Narasimhan type (HN-type for short)} of $E$.  The moduli space of vector bundles of HN- type $\sigma =(\mu (E_1))$ is the moduli space of semistable bundles. It was constructed by Mumford \cite{mun} in 1960’s using Geometric Invariant Theory and by M.S. Narasimhan and S. Seshadri  \cite{ns} using representation theory. For a treatment of a more general case we refer the reader to \cite{sim}, \cite{geiseker}, \cite{maruyama1} and  \cite{maruyama2}. The moduli space of some non-simple semistable  vector bundles with a fixed algebra of endomorphisms was constructed in \cite{yo2}, \cite{yo3} and \cite{yo4}.

Denote by  $U_{{\mu _1}}(n,d)$ the set of indecomposable vector bundles of rank $n$ and degree $d$ of coprime-type $\sigma =(\mu _1, \mu _2)$ (see Definition \ref{definitions}). If $\mu _1-\mu _2 > 2g-1$, $U_{{\mu _1}}(n,d)= \emptyset $ (see Proposition \ref{propfm}). For $0<\mu _1-\mu _2 \leq  2g-1$, $U_{{\mu _1}}(n,d)$ has a projective scheme structure that makes it an moduli scheme (see Theorem \ref{teo2}).\footnote{At the time of writing this article, some results on vector bundles of type $\sigma =(\mu _1, \mu _2)$ were obtained independently, and by different methods, in  \cite{vickyjosua} and \cite{josua}}

Our purpose is to use one more numerical invariant to describe a stratification of $U_{{\mu _1}}(n,d)$ such that each stratum is an algebraic variety, and a coarse moduli space. Such invariant is the dimension, as a $\mathbb{C}$-vector space, of the algebra of global endomorphisms.
First we determine the structure of the algebra of endomorphism of vector bundles in $U_{{\mu _1}}(n,d)$. If $\mu _1=\frac{d_1}{n_1}$, write $d_2=d-d_1$ and $n_2=d-d_1.$ We prove that (see Corollary \ref{cor2})
$$ \mbox{if} \ \  E\in U_{{\mu _1}}(n,d)\ \  \mbox{then} \ \ End(E)= \mathbb{C}[x_1,\dots , x_k]/(x_1,\dots , x_k)^2$$
where $0\leq k \leq \frac{d_1n_2-d_2n_1}{2}+n_1n_2.$

Set $A_k:=\mathbb{C}[x_1,\dots , x_k]/(x_1,\dots , x_k)^2$ and $A_0:= \mathbb{C}$. For any integer $0\leq k \leq  \frac{d_1n_2-d_2n_1}{2}+n_1n_2$ we will denote by $U_{{\mu _1}}(n,d,k)$ the set
$$U_{{\mu _1}}(n,d,k):=\{ E\in U_{{\mu _1}}(n,d): End(E)\cong A_k\}.$$

Fixing $\dim End(-)$ as invariant we prove that a flattening stratification of an $\it{Ext}^1$-sheaf over a convenient variety $Y$ (see Theorems \ref{teo1}) gives the existence of a schematic stratification of the variety $Y$ with a universal property. We use such
sub schemes and the ideas of twisted Brill-Noether theory (see \cite{hitching}) to give an algebraic structure to $U_{{\mu _1}}(n,d,k)$  and prove the following theorems. By $B^{k}(\mathcal{U}_1,\mathcal{U}^*_2)$ we mean the twisted Brill-Noether locus of product of two stable bundles with at least $k$ section and by $\mathcal{Y}_k\subset B^{k}(\mathcal{U}_1,\mathcal{U}^*_2)$ those with exactly $k$ sections (see (\ref{support}) and Section \ref{moduli} for the definition of the twisted Brill-Noether locus $B^{k}(\mathcal{U}_1,\mathcal{U}^*_2)$ and of the number $h^1$).

\begin{theorem} (Theorems \ref{teo2} and Corollary \ref{corprin0}) \   $\mathcal{U}_{{\mu _1}}(n,d,k)$  is a coarse moduli space and $\mathcal{U}_{{\mu _1}}(n,d,0)$ is a fine moduli space. Moreover, if $\mathcal{Y}_k$ is irreducible and smooth of dimension $\rho$, then $U_{{\mu _1}}(n,d,k)$ is irreducible and smooth of dimension $\rho + h^1-1.$ In particular, if $U_{{\mu _1}}(n,d,k)$ is non-empty, $B^{k}(\mathcal{U}_1,\mathcal{U}^*_2)$ is non-empty.
\end{theorem}

\smallskip

\begin{corollary} (Corollary \ref{corprin}) If $\mathcal{Y}_k$ is irreducible and smooth then $H^i(\mathcal{U}_{{\mu _1}}(n,d,k),\mathbb{C})\cong H^i(\mathcal{Y}_k, \mathbb{C})$ for $i\geq 0$.
\end{corollary}

Non-emptiness and topological properties of $\mathcal{U}_{{\mu _1}}(n,d,k)$  are given in the following theorems. Of particular interest are the vector bundles of HN-type $\sigma =(\frac{d-a}{n-1},a)$, where $a$ is an integer. That is, indecomposable unstable bundles that are extensions of a line bundle by a semistable bundle. In this case the results are a reformulation of the known results of the Brill-Noether theory
in terms of the moduli of unstable bundles. The expected dimension $\beta (g,n_1,d_1,n_2,d_2,k)$ of $U_{{\mu _1}}(n,d,k)$ is given in Section \ref{moduli}. Recall that the {\it Brill-Noether loci} are defined as
$$B(n,d,k):=\{G\in M(n,d): h^0(G)\geq k \},$$
where $M(n,d)$ is the moduli space of stable vector bundles of degree $d$ and rank $n$ over $X$.

\begin{theorem} (Theorem \ref{teop3}) Assume $0<d-an<2(n-1)$  and $(n-1,d-an,k)\ne (n-1,n-1,n-1)$. Then for $\mu _1=\frac{d-a}{n-1}$, $U_{{\mu _1}}(n,d,k)$ is non-empty if and only if $k\leq n-1+\frac{d-n(a+1)+1}{g}$. Moreover, if $U_{{\mu _1}}(n,d,k)$ is non-empty then it is irreducible and smooth of the expected dimension.
\end{theorem}

There are special results for general and Petri curves;

\begin{theorem} (Theorem \ref{teop4})
 Let  $(g,n-1,d-na,k)$  be integers that satisfies the conditions given in  Theorem \ref{bn},(5),(6) and (7).
 For general curve,  $U_{{\mu _1}}(n,d,k)$ is non-empty and has an irreducible component of the expected dimension. Moreover, if  $X$ is a Petri curve of genus $g\geq 3,  n\geq 5$ and $g\geq 2n-4$ then $U_{{\mu _1}}(n,d,k)$ is non-empty.
\end{theorem}

As in the Brill-Noether theory for vector bundles, it is possible that for
special curves  the above conclusion not holds, and where the Brill-Noether locus, and hence the moduli space, is not even reduced.
Thus, the above results differ from the corresponding results for the moduli space of
stable bundles, where non-emptiness, dimension etc. are independent of the curve.

To our best knowledge the following theorems give also new results in the Brill-Noether and twisted Bril-Noether theory (see \cite{tbn}).

\begin{theorem}(Theorem \ref{teop05})
  Assume that $B(n_1,d_1,n_1+a)$ is non-empty with $a>0$. If $2n_1<d_1<a(g+1)$ and $d_2> 2gm$ then for any $0\leq k\leq (d_2+m(1-g))(n_1+a) -(d_2n_1+d_1m +mn_1(1-g))$,   $\mathcal{Y}_k\subset B^k(U_1,U_2^*)$ is non-empty. Moreover, if $\mu_1 = \frac{d_1}{n_1}$ then $U_{{\mu_1}}({n},{d},k)$ is non-empty, where $n=n_1+n_2$ and $d=d_1+d_2$.
\end{theorem}

For Perti curves we prove

\begin{theorem}(Theorem \ref{teopetrif})  Let $X$ be a Petri curve of genus $g\geq 3$ and $(\mathcal{O}(D), V)$ a general generated linear system of degree $d_2\geq g+1$ and $\dim V =n_2+1$ with $n_2\leq 4 $ or if $n_2\geq 5$ then $g \geq  2n_2-4$. Assume that $B(n_1,d_1,t)$ is non-empty and $\frac{d_2}{n_2}<\frac{d_1}{n_1}$. For any $0\leq k \leq n_2t-n_1d_2$, $ \mathcal{Y}_k \subset B^k(U_1,U_2^*)$ is non-empty and if $n=n_2+n_1$, $d=d_2+d_1$ and $\mu _1=\frac{d_1}{n_1}$,  $U_{{\mu_1}}(n,d,k)$, is non-empty.
\end{theorem}

In the remainder of the last section we give necessary and sufficient conditions for $U_{{\mu _1}}(n,d,k)$ be smooth on $E \in U_{{\mu _1}}(n,d,k)$.
Let $0\subset E_1\subset E$ be the HN-filtration of $E$ in $U_{{\mu _1}}(n,d,k)$. Write $E/E_1=F_1$.  We use the following diagram

$$
\begin{array}{ccc}
H^1(End(E_1))\oplus  H^1(End(F_1))&\stackrel{d\Phi}{\longrightarrow } & H^1(End(E_1\otimes F^*_1))\\
&\eta _E\searrow & \beta\downarrow \\
&&H^0(E_1\otimes F^*_1)^*\otimes H^1(E_1\otimes F^*_1)
\end{array}
$$
to define $\eta _E$ as  $\eta _E:=\beta \circ d\Phi $ and prove

\begin{theorem}( Theorem \ref{teop8}) $U_{{\mu _1}}(n,d,k)$ is smooth at $E$ and of the expected dimension
if and only if $\eta _E $ is surjective.
\end{theorem}

In Section \ref{unstables} we review some of the standard
facts on unstable bundles. Section \ref{algebra} will be concerned with the algebra of endomorphisms of unstable bundles. In Sections \ref{moduli} and \ref{nonempty} our main results are stated and proved.

{\bf Acknowledgments:} The first author gratefully acknowledges the many helpful suggestions of Peter E. Newstead
during the preparation of the paper. Her thanks are also to Alfonso Zamora Saiz for drawing the author’s attention to some unclear points.
 The first author is a member of the international research group VBAC (Vector Bundles on Algebraic Curves).

{\bf Notation} The rank and degree of a vector bundle $E$ are denoted by $rk(E)$ and $d(E)$ respectively and the slope as the rational number $\mu (E):=\frac{d(E)}{rk(E)}$. We will write the projection in the i-factor as $p_i$ and by $<E\stackrel{f}{\to} F>$ the linear space generated by the function $f:E\to F$.  The cohomology groups $H^i(X,E)$ as $H^i(E)$ and its dimension as $h^i(E)$. Given an exact sequences
$$
\rho: 0\to G \to E \to  F \to 0,
$$
\begin{itemize}
\item by $\rho ^*$ we mean the dual sequence $\rho ^*: 0\to F^* \to E^* \to G^* \to 0,$
\item by  $(M\otimes \rho )$ the sequence $\rho $ tensor by the vector bundle $M$ i.e.
$$
(M\otimes \rho): 0\to M\otimes G \to M\otimes E \to M\otimes F \to 0.
$$
\item By $H^* (\rho)$ we mean the cohomology sequence of $\rho $
$$ 0\to H^0(X,G) \to H^0(X,E) \to H^0(X,F) \stackrel{\delta}{\to} H^1(X,G) \to H^1(X,E) \to  \dots .$$
\item To shorten notation, sometimes we write $E\in H^1(X,F^*\otimes G)$ instead of $(\rho: 0\to G \to E \to  F \to 0) \in H^1(X,F^*\otimes G)$
\end{itemize}
\section{Unstable bundles of HN-lenght 2}\label{unstables}

From now on, $X$ will be a smooth irreducible complex projective curve of genus $g \geq  2$. Recall that a vector bundle $E$ over $X$ is semistable if for all proper subbundle $F\subset E$ the slopes satisfy the following inequality
$$\mu(F)\leq \mu (E).$$
 The vector bundle $E$ is stable if the inequality is strict and unstable if it is not semistable.

In \cite{hn} it was proved that any vector bundle over $X$ has a unique filtration,  called {\it Harder-Narasimhan filtration},
\begin{equation}
 0 = E_0\subset E_1 \subset \cdots \subset E_m = E
 \end{equation}
 such that for $1 \leq i \leq  m,$
\begin{itemize}
  \item $E_i/E_{i-1}$ is semistable and
  \item \begin{equation}\label{desigualdadmu} \mu (E_1) > \mu (E_2/E_1) > \cdots > \mu (E_m/E_{m-1}).
\end{equation}
\end{itemize}
To shorten notation, we write HN-filtration instead of Harder-Narasimhan filtration and for abbreviation, we write $F_i$ instead of the quotient $E_{i+1}/E_i$, when no confusion can arise. Note that $F_0=E_1$.
The HN-max and HN-min of $E$ are defined as $$\mu _{max}(E):=\mu (E_1)\  \ \mbox{and} \ \ \mu _{min}(E):=\mu (F_{m-1}). $$ The vector bundle $E_1$ is called {\it the maximal destabilizing subbundle}.

The following definitions will be used throughout all the paper.

\begin{definition}\begin{em}\label{definitions}
Let $ 0 = E_0\subset E_1 \subset \cdots \subset E_m = E$ be the Harder-Narasimhan filtration of $E$.
\begin{itemize}
\item The HN-type is the sequence of slopes $\sigma=(\mu (E_1), \dots ,\mu( E_m/E_{m-1})).$
\item The number $m$ is called {\it the HN-length} of the HN-filtration.
\item The HN-filtration is called of {\it simple type} if each  $E_i$ and $F_{i}$ are simple for $i=1,\dots ,m-1$.
\item The HN-filtration  is called of {\it coprime type} if the numbers in each slope $\mu (F_{i})$ and $\mu (E_i)$ are coprime for $i=1,\dots ,m-1$.
\item The HN-filtration  is called {\it HN-indecomposable} if  each $E_i$ and $F_i$ is indecomposable $i=1,\dots ,m-1$.
\item $E$ is called {\it $HN$-general} (respectively {\it HN-special})  if all the quotient $F_i$ are general (respectively  special) in the Brill-Noether theory.
\item The extension $\rho _i: 0\to E_{i} \to E_{i+1} \to F_{i} \to 0$ is called the HN(i)-extension and the sequence of extensions $\rho=(\rho _1, \cdots, \rho _{m-1})$ is called the {\it HN-sequences of $E$}.
\end{itemize}
\end{em}
\end{definition}

\begin{remark}\begin{em}\label{remigual}
Let  $ 0 = E_0\subset E_1 \subset \cdots \subset E_m = E$ be the HN-filtration of $E$. Note that:
\begin{enumerate}
\item the condition to being of simple/coprime type or HN-indecomposable is for $i=1,\dots ,m-1$. Therefore, $E_m=E$ is not necessarily simple or indecomposable. 
 \item  Simple type $\Longrightarrow$ HN-indecomposable type. However, HN-indecomposable $\nRightarrow$ simple type.
\item  If $i\ne 0$,  $E_i$ is unstable and $ 0 = E_0\subset E_1 \subset \cdots \subset E_i$  is the HN-filtration of $E_i$. Moreover, if the filtration of $E$ is of coprime, simple, HN-indecomposable,  or general type, the same holds for the HN-filtration of any $E_i.$ This will allow us to make induction on the HN-length.
     \item For each $E_i$ we have an exact sequence
\begin{equation}\label{extension}
\rho _i: 0\to E_{i-1} \to E_i \to F_{i-1} \to 0.
\end{equation}
Thus, $E$ is constructed by a successive sequence of extensions.
\end{enumerate}
\end{em}\end{remark}

In this section we will restrict our attention to the case of vector bundles of HN-type $\sigma =(\mu _1, \mu _2)$. Note that the value of $\mu _1$ fix the value of $\mu _2$ and viceversa. Indeed, if $\mu _1=\frac{d_1}{n_1} $ then $\mu _2=\frac{d(E)-d_1}{rk(E)-n_1}$. Our aim is to describe some properties and parameterize vector bundles of HN-type $\sigma =(\mu _1, \mu _2)$.

Let
\begin{equation}\label{extension tipo2}
(\rho _1: 0\to E_{1} \stackrel{i}{\to} E_2 \stackrel{p}{\to}  F_1 \to 0) \in Ext^1(F_1,E_1)\cong H^1(X,F_1^*\otimes E_1).
\end{equation} be an extension of two semistable bundles with $\mu_1:=\mu(E_1)>\mu (F_1)$.
It follows that $E_2$ is unstable and $0\subset E_1\subset E_2$ is its HN-filtration, with $E_2/E_1=F_1$.

\begin{remark}\begin{em}\label{rem1} Note that given two extensions $(\rho _1: 0\to E_{1} \stackrel{i}{\to} E_2 \stackrel{p}{\to}  F_1 \to 0)$ and $(\rho '_1: 0\to E_{1} \stackrel{i}{\to} E_2' \stackrel{p}{\to}  F_1 \to 0)$ in $Ext^1(F_1,E_1)$, $E_2\cong E_2'$ if and only if $ \rho _1\sim \rho _1 '$ in $Ext^1(F_1,E_1)/(Aut(E_1)\times Aut(F_1))$. Moreover, if $E_1$ and $F_1$ are simple, $E_2\cong E_2'$ iff $\rho =\lambda \rho '$ with $\lambda \in \mathbb{C}^*$.
\end{em}\end{remark}

To prove our first results we use the following lemmas.

\begin{lemma}\label{lema01}\label{dim1} Let
$\rho _1: 0\to E_{1} \stackrel{i}{\to} E_2 \stackrel{p}{\to}  F_1 \to 0$
 be an extension of two semistable bundles with $\mu_1:=\mu(E_1)>\mu (F_1)$. If $\mu (E_1)-\mu(F_1)> 2g-2$, $\rho _1=0$.
 Moreover, if $\rho _1$ is non-trivial then
\begin{enumerate}
\item $H^0(X,E^*_1\otimes E_1)=H^0(X,E^*_1\otimes E_2 )$. Moreover, if  $E_1$ is simple, $h^0(E^*_1\otimes E_2)=1$ and $H^0(E^*_1\otimes E_2)=<E_{1} \stackrel{i}{\to} E_2>$.
\item $H^0(F_1^*\otimes F_1)= H^0(E_2^*\otimes  F_1)$. Moreover, if $F_1$ is simple, $h^0(E_2^*\otimes  F_1)=1 $ and $H^0(E_2^*\otimes  F_1)=<E_2 \stackrel{p}{\to}  F_1>$.
\item  If $F_1$ is simple then  $H^0(F_1^*\otimes E_1)= H^0(F_1^* \otimes E_2 )$.
\end{enumerate}

\end{lemma}

\begin{proof}
 The inequality  $\mu (E_1)-\mu(F_1)> 2g-2$ implies that $\mu (F_1^*\otimes E_1)>2g-2 $. Hence, since $F_1^*\otimes E_1$ is semistable, $H^1(X,F_1^*\otimes E_1)=0$.

Since  $H^0(X,E_1^*\otimes F_1)=0$, $(1)$ and $(2)$ follow from the cohomology sequences
\begin{equation}\label{eq1}
H^*(E_1^*\otimes (\rho _1 )): 0\to H^0(X, E_1^*\otimes E_1)\to H^0(X,E_1^*\otimes E_2)\to H^0(X, E_1^*\otimes F_1)\stackrel{\delta _0}{\to}  \cdots
\end{equation}
and
\begin{equation}\label{eq3}
H^*((\rho _1 )^*\otimes F_1): 0\to H^0(X,F_1^*\otimes F_1)\to H^0(X,E_2^*\otimes  F_1)\to H^0(X, E_{1}^*\otimes  F_1)\stackrel{\delta _1}{\to} \cdots
\end{equation}

$(3).-$ From the cohomology sequence $H^*(F_1^*\otimes (\rho _1 ))$
\begin{equation}\label{eq2}
 0\to H^0(X,F_1^*\otimes E_1)\to H^0(X,F_1^* \otimes E_2 )\to H^0(X, F_{1}^*\otimes  F_1)\stackrel{\delta _2}{\to} H^1(X,F_1^*\otimes E_1) \cdots
\end{equation}
$H^0(X,F_1^*\otimes E_1)= H^0(X,F_1^* \otimes E_2 )$, since $F_1$ is simple and $\rho _1\ne 0$.

\end{proof}

 \begin{lemma}\label{indecom} Assume $E_1$ and $F_1$ are semistable with $E_1$ simple and $F_1$ indecomposable. Then $E_2$ is indecomposable if and only if $\rho _1 \ne 0$.
\end{lemma}

\begin{proof} One direction is trivial. Assume $\rho _1\ne 0.$
Suppose, contrary to our claim, that $E_2=G_1\oplus G_2$ and the extension
$$0= \rho : 0\to G_1\stackrel{j}{\to} E_2\stackrel{q}{\to} G_2\to 0$$ splits. Denote by  $\gamma :G_2\to E_2$ the splitting morphism, i.e.  $q\circ \gamma = id_{G_2}$.  We can certainly assume that $\mu (G_2) \leq \mu (E_2)$, for if not, we replace $G_2$ by $G_1$.

The inclusion $i:E_1\to E_2$ in the extension $\rho _1$ induces the following diagram
$$
\begin{array}{ccccccccc}
&& 0 && 0&& 0&& \\
&& \downarrow& & \downarrow&& \downarrow&& \\
0&\to& M& {\to}& E_1&\stackrel{\sigma}{\to}& H&\to & 0\\
&& \downarrow& & \downarrow i&& \downarrow \nu&& \\
0&\to &G_1&\stackrel{j}{\to} & E_2 &\stackrel{q}{\rightleftarrows}_{\gamma} & G_2& \to& 0
\end{array}
$$

We have different cases.

Case (1)-. $H\ne 0$ and $M\ne 0.$

In this case there is a contradiction to Lemma \ref{lema01},(1) since the morphism  $\gamma\circ \nu \circ \sigma :E_1\to E_2$ is not a scalar multiple of the inclusion $i:E_1\to E_2$.

Case (2)-. $H=E_1$.

The diagram
$$
\begin{array}{ccccccccc}
&&  && 0&& 0&& \\
&& & & \downarrow&& \downarrow&& \\
&& & & E_1&=& E_1& & \\
&&& & \downarrow i&& \downarrow && \\
0&\to &G_1&\stackrel{j}{\to} & E_2 &\stackrel{q}{\rightleftarrows}_{\gamma} & G_2& \to& 0
\end{array}
$$

induces the following diagram

$$
\begin{array}{ccccccccc}
&&  && 0&& 0&& \\
&& & & \downarrow&& \downarrow&& \\
&& & & E_1&=& E_1& & \\
&&& & \downarrow i&& \downarrow && \\
0&\to &G_1&\stackrel{<}{\to} & E_2 &\stackrel{<}{\rightleftarrows}_{\gamma} & G_2& \to& 0 \\
&& \|& & \downarrow \vee && \downarrow&& \\
0&\to& G_1& \stackrel{\leq }{\to}& F_1&\stackrel{\sigma}{\to}& T&\to & 0\\
&& \downarrow& & \downarrow && \downarrow && \\
&& 0 && 0&& 0&&.
\end{array}
$$

The semistability of $F_1$ implies that $\mu (G_1) \leq \mu(F_1)< \mu (E_2)$, hence that $\mu (E_2)< \mu (G_2)$, which is a contradiction to our assumption.

Case (3)-. $H=0$.

In this case we have the following diagram

$$
\begin{array}{ccccccccc}
 && 0&& 0&& &&  \\
 & & \downarrow&& \downarrow&& &&  \\
 & & E_1&=& E_1& & && \\
& & \downarrow i&& \downarrow i && && \\
0&\to &G_1&\stackrel{j}{\to} & E_2 &\stackrel{q}{\rightleftarrows}_{\gamma} & G_2& \to& 0 \\
&& \downarrow& & \downarrow p && \|&& \\
0&\to& T_1& \stackrel{\jmath }{\to}& F_1&\stackrel{\ell}{\to}& G_2&\to & 0\\
&& \downarrow& & \downarrow && \downarrow && \\
&& 0 && 0&& 0&&.
\end{array}
$$

If $0\ne \eta := p\circ \gamma : G_2\to F_1$, $$\ell\circ \eta = \ell\circ p\circ \gamma =q\circ \gamma =id _{{G_2}}$$ and hence $F_1$ is decomposable which is a contradiction, since $F_1$ is indecomposable.

If $\eta =0$ then $G_2\subset E_2\subset G_1$, which is also a contradiction since $E_2=G_1\oplus G_2$.

Therefore, $E_2$ is indecomposable, as claimed.

\end{proof}

Let $T=Ext^1(F_1,E_1)=H^1(X,F_1^*\otimes E_1)$. It is well known (see \cite{nr}) that $\mathbb{P}(T)$ parameterize non-trivial extensions  $\rho _1: 0\to E_{1} \stackrel{i}{\to} E \stackrel{p}{\to}  F_1 \to 0$ and that if $\mathbb{H}$ is the hyperplane bundle over $\mathbb{P}(T)$ the extension
$$ 0\to p_1^*( E_1)\otimes p_2^*(\mathbb{H}) \to \mathcal{E} \to p_1^*( F_1)\to 0 $$
that corresponds to the identity under the isomorphism $$ End(T)\cong H^1(X\times \mathbb{P}(T), p_1^*(Hom(F_1,E_1))\otimes p_2^*(\mathbb{H}))$$ has universal properties. Thus, Lemma \ref{indecom} leads us to the following result.

\begin{proposition}\label{propfm} If $E_1$ and $F_1$ are simple semistable vector bundles with $\mu (E_1)>\mu (F_1)$ then the pair $(\mathbb{P}(T),  \mathcal{E})$ is a fine moduli space for
 indecomposable unstable bundles $E$ with $E_1$ as the maximal destabilizing subbundle of $E$ and $F_1=E/E_1$. Moreover, if $\mu (E_1)-\mu(F_1)> 2g-2$, $\mathbb{P}(T)= \emptyset$.

 \end{proposition}

 From now on we tacitly assume that $0<\mu (E_1)-\mu(F_1)\leq 2g-2.$ By Riemann-Roch
\begin{equation}\label{dimension}
\dim \mathbb{P}(T)=h^0(F_1^*\otimes E_1)-\tilde{d}+\tilde{r}(g-1)-1 ,
\end{equation}
where $\tilde{d}:= deg(E_1)rk(F_1) -deg(F_1)rk(E_1)$ and $\tilde{r}:=rk(E_1)rk(F_1)$.
The dimension of $\mathbb{P}(T)$ depends on $h^0(F_1^*\otimes E_1)$. Actually, $\dim \mathbb{P}(T)>0$ if and only if $h^0(F_1^*\otimes E_1)>\tilde{d}+\tilde{r}(1-g)+1$. Since $F_1^*\otimes E_1$ is semistable and $0<\mu (E_1)-\mu(F_1)\leq 2g-2 $, the study of the dimension $h^0(F_1^*\otimes E_1)$ belongs to the Brill-Noether Theory for vector bundles.

We will touch only a few aspects of the theory. Recall that a vector bundle $G$ is special if $h^0(G)\cdot h^1(G)\ne 0$, otherwise is called general. The {\it Brill-Noether loci} are defined as
$$B(n,d,k):=\{G\in M(n,d): h^0(G)\geq k \},$$
where $M(n,d)$ is the moduli space of stable vector bundles of degree $d$ and rank $n$ over $X$. Moreover, the Brill-Noether loci induce the following filtration:
$$M(n,d)\supseteq B(n,d, 1)\supseteq \cdots \supseteq B(n,d, k) \supseteq B(n,d,k+1) \supseteq \cdots. $$
The number $$\rho (g,n,d,k):= n^2(g-1)+1 -k(k-d+n(g-1))$$ is called {\it the Brill-Noether number} and is the expected dimension of $B(n,d,k)$. The Brill-Noether loci are also defined on the moduli space of $S$-equivalent semistable bundles (see \cite{bgn}).

\begin{remark}\begin{em}\label{remigualk}
 Denote by $\chi $ the number $\chi :=d+n(1-g)$.
Note that if $k\leq \chi$,  $B(n,d,k)= M(n,d)$. Thus, to have proper Brill-Noether loci we will assume that $k> \chi$.
\end{em}\end{remark}

For the convenience of the reader we summarize the main results of these concepts and the relevant material for us
without proofs, thus making our exposition self-contained. See \cite{im} for a survey of the main results on the Brill-Noether Theory.  Recall that $X$ is a {\it Petri curve} if the multiplication map $\mu : H^0(X,L)\otimes H^0(X, K_X\otimes L^*) \to H^0(X, K_X)$ is injective for every line bundle $L$ on $X$ (see \cite{arb}).

\begin{theorem}\label{bn} Let $G$ be a semistable vector bundle of rank $n$ and degree $d$ on $X$.
\begin{enumerate}
\item If $G$ is general,
$$
h^0(G) =
\left\{ \begin{array}{ll}
          0 & \text{if $ 0<\mu(G) <g-1$,}\\
         {d}+{r}(1-g) & \text{if $g-1 \leq \mu(G) <2(g-1)$.}
          \end{array}\right.
$$
\item If $B(n,d,k)$ is proper, the general point $E$ of every proper component of $B(n,d,k)$ has $h^0(E)=k$.
\item (Clifford's Theorem for vector bundles \cite[Theorem 2.1]{bgn}) If $G$ is special, $$h^0(G)\leq \frac{d}{2} +n.$$
\item  (\cite{bgn} and \cite{mer}) For any curve $X$, if $0<\mu(G)<2$,
$$B(n,d,k)\ne \emptyset  \ \mbox{ if and only if} \  n<d+(n-k)g \ \mbox{ and} \  (n,d,k)\ne (n,n,n).$$
Moreover, if $B(n,d,k)$ is not empty, then it is irreducible, of dimension $\rho (g,n,d,k)$ and $Sing(B(n,d,k))= B(n,d,k+1)$.

\item (\cite{bmno}) For general curve $B(n,d,k)$ is not empty if $d = nd'+ d''$ and $ n\leq d'' +(n-k)g$ and the following conditions are satisfied $$0 < d'' < 2n, \  d' \geq  \frac{(s-1)(s+g)}{s}, \  \mbox{with} \ 1 < s < g, \  \ (d'', k)\ne (n,n).$$
\item (\cite{im}) If $X$ is generic and $d = nd_1+d_2, \ k =nk_1+k_2, \ d_2 < n, \ k_2 < n$ for some positive integers $d_i$ and $k_i$ then $B(n,d,k)\ne \emptyset $ and has one component of the expected dimension if one of the following conditions is satisfied:
    $$
    \begin{array}{ll}
g-(k_1+1)(g-d_1+k_1-1)\geq 1, & 0=d_2 \geq  k_2 \\
g-k_1(g-d_1+k_1+1) > 1, & d_2=k_2=0 \\
 g-(k_1+1)(g-d_1+k_1) \geq  1, & d_2 < k_2.
 \end{array}
 $$
 \item (\cite{bbn2}) Let $X$ be a Petri curve of genus $g\geq 3,  n\geq 5$ and $g\geq 2n-4$. Then $B(n,d,n+1)$ is non-empty.
\end{enumerate}
\end{theorem}

\begin{remark}\begin{em} The non-emptiness of the locus $B(n,d,n+1)$ has been proved with different relation among $n,d$ and $g$ (see \cite{bbn2},  \cite{bbn1} or \cite{b} and the bibliography there). For simplicity we just quote one of the results in Theorem \ref{bn},(7).
\end{em}\end{remark}

According to the above theorem, in some cases the non emptiness and dimension of the moduli space $\mathbb{P}(T)$
differ from general or special curves.
Theorem \ref{bn} will be relevant in Section \ref{nonempty}  when we consider families of extensions of semistable bundles.

The remainder of this section will be devoted to the general properties of unstable bundles.

It is well know that $H^0(E)=0$ for semistable vector bundles $E$ of negative degree. However, unstable bundles of negative degree can have sections. Indeed,  let $F$ be a bundle of positive degree with $H^1(F)\ne 0$. Any extension $\rho:  0\to \mathcal{O} \to G \to F^* \to 0 \in H^1(F)$ defines an unstable bundle $G$ with at least one section. In the following proposition we proof for unstable bundles some of the well known properties of semistable bundles.  They follow directly from the properties of the HN-filtration. Maybe some are already known but we prefer to include them in one proposition.

\begin{proposition}\label{propunstable} Let $E$ be an indecomposable unstable vector bundle of rank $n$ and degree $d$ and $ 0 = E_0\subset E_1 \subset \cdots \subset E_m = E$ its Harder-Narasimhan filtration.
\begin{enumerate}
\item If $\mu _{min}(E)>2g-1$ then $H^1(X,E)=0$ and $E$ is (globally) generated. Moreover, indecomposable unstable bundles with fix HN-filtration are bounded.
\item If $E$ is HN-special, $h^0(E)\leq \frac{d}{2}+n$.
\item If $E$ is HN-general and $\mu _{min}(E)>g+1$, $H^1(X,E)=0$ and $h^0(E)= d(E)+rk(E)(1-g)$. Moreover, if $0<\mu _{max}(E) <g, \  h^0(E)=0.$
\end{enumerate}
\end{proposition}

\begin{proof}
 Let $\rho _i: 0\to E_{i-1} \to E_i \to F_{i-1} \to 0$ be the HN(i)-extension of $E$ and
\begin{equation}\label{cohom0}
0\to H^0(X,E_{i-1}) \to H^0(X,E_i) \to H^0(X,F_{i-1}) \stackrel{\delta}{\to}
\end{equation}
\begin{equation}\label{cohom1}
\stackrel{\delta}{\to} H^1(X,E_{i-1}) \to H^1(X,E_i) \to H^1(X,F_{i-1}) \to 0
\end{equation}
its cohomology sequence. Recall that we have the inequalities $$\mu _{max}(E)=\mu (E_1) > \mu (F_1) > \cdots > \mu (F_{m-1})= \mu _{min}(E)>2g-1.$$
The proofs are by induction on the HN-length of $E$. We give the proof for the case $m=2$ using the  HN(2)-extension.

$(1)$.- Since $\mu (E_1) > \mu (F_1)= \mu _{min}(E_2)>2g-1$, (\ref{cohom1}) shows that $H^1(E_2)=0$, by the semistability of $E_1$ and $F_1$. As $E_1$ and $F_1$ are generated we have that $E_2$ is generated. That are bounded follows from \cite{maruyama1} (see also \cite{hk} ).

$(2)$.- We conclude from (\ref{cohom0}) that  $h^0(E_2)\leq h^0(E_1) +h^0(F_1), $  hence that $$ h^0(E_1) +h^0(F_1)\leq \frac{d(E_1)}{2}+rk(E_1)+\frac{d(F_1)}{2}+rk(F_1)$$ (see Theorem \ref{bn} (2))  and finally that  $$h^0(E_2)\leq  \frac{d(E_2)}{2}+rk(E_2).$$

 $(3)$.- Follows also from (\ref{cohom0}) and (\ref{cohom1}), Riemann-Roch and the assumption that HN-filtration is HN-general.

From Remark \ref{remigual} we can apply induction. Analysis similar to that in the proof of $m=2$ (now using the  HN(i)-extension), shows that $(1),(2)$ and $(3)$ are satisfied for any $m>2$. The details are left to the reader.

\end{proof}

It is known (see \cite[Lemma 1.3.3]{danle}) that if $E$ and $E'$ are two unstable vector bundles with $\mu_{min}(E)>\mu_{max}(E')$ then $Hom(E,E')=0$. Our aim is to study $Hom(E,E)=End(E)$ for unstable bundles.

\section{Algebra of endomorphisms}\label{algebra}

In this section, we determine the structure of the algebra of endomorphisms $End(-)$ of a HN-indecomposable vector bundle of rank $n$, degree $d$ and of HN-length $2$. Using the HN-filtration, we give also an upper bound for $\dim End(-)$ when the HN-length $>2$.

Recall from  \cite{atiyah}, that if $E$ is indecomposable vector bundle then
\begin{equation}\label{nil}
End(E)\cong (Id_E)\oplus Nil(E),
\end{equation}
where $Nil(E)$ is the subset $Nil(E)\subset End(E)$ consisting of all nilpotent global endomorphisms of $E$. In \cite{yo4}, the following upper bound
\begin{equation}\label{dimalgs}
\dim End(E)\leq 1+\frac{n(n-1)}{2}
\end{equation}
was given for indecomposable semistable bundles of rank $n$. The bound for indecomposable bundles of HN-length $2$ is established by our next proposition.

\begin{proposition}\label{propalg} Let $E_2$ be a indecomposable unstable bundle and $ 0 = E_0\subset E_1 \subset E_2$ be its HN-filtration. Then, if $E_1$ and $F_1=E_2/E_1$ are indecomposable then
\begin{equation}\label{dimalg2}
\dim End(E_2)\leq 1+\frac{n(n-1)}{2} +  h^0(F_1^*\otimes E_1),
\end{equation}
where $F_1=E_2/E_1$.
\end{proposition}

\begin{proof} Let  $\rho _1 : 0\to E_1 \stackrel{\iota}{\to} E_2\stackrel{p}{\to} F_1 \to 0$ be the HN(2)-sequence of $E_2$. Assume $n_1:=rk(E_1)$ and $n_2:=rk(F_1)$.
We will use the following sequences
\begin{equation}\label{eqalg1}
 0\to H^0(X, E_1\otimes E_2^*)\to H^0(X, E_2\otimes E_2^*)\to H^0(X, F_1\otimes E_2^*)\stackrel{\delta _0}{\to} H^1(X, E_1\otimes E_2^*) \cdots
\end{equation}
\begin{equation}\label{eqalg2}
 0\to H^0(X, F_1^*\otimes F_1)\to H^0(X, E_2^*\otimes  F_1)\to H^0(X, E_{1}^*\otimes  F_1)\stackrel{\delta _1}{\to} H^1(X, F_1^*\otimes F_1) \cdots
\end{equation}
\begin{equation}\label{eqalg3}
 0\to H^0(X,F_1^*\otimes E_1)\to H^0(X,E_2^* \otimes E_1 )\to H^0(X, E_{1}^*\otimes  E_1)\stackrel{\delta _2}{\to} H^1(X,F_1^*\otimes E_1)
\end{equation}
that are part of the cohomology sequences of $H^*((\rho _1 )\otimes E_2^*)$, $H^*((\rho _1 )^*\otimes F_1)$ and of $H^*((\rho _1 )^*\otimes E_1)$ , respectively.

From (\ref{eqalg1})
\begin{equation}\label{alg2}
\dim End(E_2)\leq h^0(E_1\otimes E_2^*) + h^0( F_1\otimes E_2^*).
\end{equation}
The semistability of the $F_1$ and $E_1$ and the inequalities (\ref{desigualdadmu}) imply that $H^0(X, E_{1}^*\otimes  F_1)=0$, hence from (\ref{eqalg2})
\begin{equation}\label{eqigualdad} H^0(X, F_1\otimes E_2^*)= H^0(X,F_1^*\otimes F_1)=End(F_1).
\end{equation}
From (\ref{eqalg3})
$$h^0( E_1 \otimes E_2^* )\leq  h^0(F_1^*\otimes E_1) + h^0( E_{1}^*\otimes  E_1) $$
since the extension $\delta_2(Id)$ is the element classifying the extension $(\rho_1)^*$ and $\rho ^*_1$ is non-trivial.
Therefore, since $E_1$ and $F_1$ are indecomposables
$$\begin{array}{ccl} \label{alg2f}
\dim End(E_2)&\leq &h^0(E_1\otimes E_2^*) + h^0( F_1\otimes E_2^*)\\
&=& h^0(E_1\otimes E_2^*) + \dim End(F_1)\\
&\leq & h^0(F_1^*\otimes E_1) + \dim End(E_1) + \dim End(F_1)\\
&\leq &   h^0(F_1^*\otimes E_1) +\frac{n_1(n_1-1)}{2} +1+\frac{n_2(n_2-1)}{2}+1 \ \ \ \  (\mbox{see} \ \ (\ref{dimalgs})) \\
& \leq& h^0(F_1^*\otimes E_1)+ 1+\frac{(n_1^2+n_2^2)}{2} -\frac{n_1+n_2}{2} +1\\
& \leq& h^0(F_1^*\otimes E_1)+ 1+\frac{n(n-1)}{2} +1 -n_1n_2\\
&\leq &  h^0(F_1^*\otimes E_1)+ 1+\frac{n(n-1)}{2}
\end{array}
$$
as claimed.
\end{proof}

\begin{corollary}\label{cor2} Let $E_2$ be a indecomposable bundle of rank $n$ of HN-length $2$.
If $E_1$ and $F_1$ are simple then $\dim End(E_2)= 1+ h^0(F_1^*\otimes E_1)$ and $$End(E_2)\cong \mathbb{C}[x_1,\dots , x_k]/(x_1,\dots , x_k)^2,$$
where $k=h^0(F_1^*\otimes E_1)$. Moreover, $E_2$ is simple if and only if $h^0(F_1^*\otimes E_1)=0$.
\end{corollary}

\begin{proof}
 The vector bundle $E_2$ is indecomposable of simple type, hence  Lemma \ref{dim1}, (2) shows that $h^0(F_1\otimes E_2^*)=1$ and from (\ref{eqalg1}) we conclude that $\delta _0(\lambda p)= 0$. Thus,
$$\dim End(E_2)=h^0(E_1\otimes E_2^*) + h^0( F_1\otimes F_1^*)= h^0(E_1\otimes E_2^*) + 1.$$
It follows from (\ref{eqalg3}) that $\delta _2(Id_{E_1})\ne 0$, hence that  $h^0(F_1^*\otimes E_1)=h^0( E_1 \otimes  E_2^*)$, and finally that $$\dim End(E_2)= 1+ h^0(F_1^*\otimes E_1).$$

From (\ref{nil}) we deduce that $\dim Nil(E)=h^0(F_1^*\otimes E_1)$.
The map $$\varphi :H^0(F_1^*\otimes E_1) \to Nil(E)$$ defined as $\sigma \mapsto \iota \circ\sigma \circ p$ is a well defined injective homomorphism. Hence, $H^0(F_1^*\otimes E_1)\cong Nil(E)$. Moreover, $(\varphi (\sigma ))^2=0$ and
$$(\varphi (\sigma _1 ))\circ (\varphi (\sigma _2))=(\varphi (\sigma _2))\circ(\varphi (\sigma _1))=0.$$
We thus get $End(E)\cong \mathbb{C}[x_1,\dots , x_k]/(x_1,\dots , x_k)^2$ where $k=h^0(F_1^*\otimes E_1)$. Moreover, $E_2$ is simple if and only if  $h^0(F_1^*\otimes E_1)=0$  which is our claim.

\end{proof}

\begin{remark}\begin{em}\label{decomposable} Note that if $F_1^*\otimes E_1$ is special $h^0(F_1^*\otimes E_1)\cdot h^1(F_1^*\otimes E_1)\ne 0$ and if  $0\ne \rho:0\to E_1\to E_2\to F_1\to 0\in H^1(X,F_1^*\otimes E_1)$, $E_2$ is indecomposable and $\dim End(E_2) >1$. Moreover, if  $F_1^*\otimes E_1$ is no special. $h^0(F_1^*\otimes E_1)\cdot h^1(F_1^*\otimes E_1)= 0$, consequently, either $E_2$ is decomposable or is simple.
\end{em}\end{remark}

The next propositions is a fairly straightforward generalization for unstable bundles of HN-length $m>2$.

\begin{proposition}\label{dimensionr} Let $E$ be an indecomposable unstable vector bundle of rank $n$. If the HN-filtration $ 0 = E_0\subset E_1 \subset \cdots \subset E_m = E$ is HN-indecomposable
then
\begin{equation}\label{dimalgr}
\dim End(E)\leq 1+\frac{n(n-1)}{2} + \sum ^{m-1} h^0(F_i^*\otimes E_{i})
\end{equation}
\end{proposition}

For the proof we will use the following lemma.

\begin{lemma}\label{lema1}\label{lemfin} For all $0<i< j\leq m $
\begin{enumerate}
\item $H^0(X,F_i^*\otimes F_j)=0$ and
\item  for all $0<i\leq  j\leq m $, $H^0(X, E_i^*\otimes F_j)=0.$
\end{enumerate}
\end{lemma}

\begin{proof} $(1)$ follows from the semistability of the quotients $F_i$'s and the inequalities (\ref{desigualdadmu}).

$(2)$ The semistability of $E_1$ and  $F_1$  and the inequalities (\ref{desigualdadmu}) imply that
\[
H^0(X,F_1^*\otimes F_j)=H^0(X,E_1^*\otimes F_j)=0,
\]
Hence, from the cohomology sequence
$$H^*((\rho _1)^*\otimes F_j): 0\to H^0(X,F_1^*\otimes F_j) \to H^0(X,E_2^*\otimes F_j) \to H^0(X,E_1^*\otimes F_j) \to \cdots $$
 we conclude that $H^0(X,E_2^*\otimes F_j)=0$.
In the same manner we can see that $H^0(X,E_i^*\otimes F_j)=0$ for all $i\leq j\leq m.$  The detailed verification being left to the reader.
\end{proof}

{\it Proof of Proposition \ref{dimensionr}}
The proof is by induction on the HN-length. The first step of induction is Proposition \ref{propalg}. From the cohomology sequence $H^*(\rho _{m-1} \otimes E_m^*)$
\begin{equation}\label{algm}
\dim End(E_m)\leq h^0(E_{m-1}\otimes E_m^*) + h^0( F_{m-1}\otimes E_m^*).
\end{equation}
We can now proceed analogously to the proof of Proposition \ref{propalg}. From the cohomology sequence $H^*((\rho _{m-1})^*\otimes F_{m-1})$ and Lemma \ref{lema1},(2) it follows that
\[
H^0( F_{m-1}\otimes E_m^*)=H^0(F_{m-1}^*\otimes F_{m-1}).
\]
 From the cohomology sequence $H^*((\rho _{m-1})^*\otimes E_{m-1})$
we have the inequality
\[
h^0(E_{m-1}\otimes E_m^*)\leq h^0(End(E_{m-1}))+ h^0(F_{m-1}^*\otimes E_{m-1}).
\]
 Recall that $E_m$ is HN-indecomposable.

Now (\ref{algm} ) becomes
$$
\begin{array}{cll}
\dim End(E_m) &\leq & h^0(End(E_{m-1}))+ h^0(F_{m-1}^*\otimes E_{m-1}) + h^0(F_{m-1}^*\otimes F_{m-1})\\
&&\\
&\leq &\frac{rk(E_{m-1})(rk(E_{m-1})-1)}{2}  +\frac{rk(F_{m-1})(rk(F_{m-1})-1)}{2} +2+\\
&& \sum ^{m-2} h^0(F_i^*\otimes E_{i}) + h^0(F_{m-1}^*\otimes E_{m-1})\\
&&\\
&\leq & 1+\frac{rk(E_{m})(rk(E_{m})-1)}{2} + \sum ^{m-1} h^0(F_i^*\otimes E_{i}),
\end{array}
$$
which is the desired conclusion.
$\hfill{\Box}$

The following result may be proved in much the same way as Corollary \ref{cor2}.

\begin{corollary}  Let $E$ be an indecomposable unstable vector bundle of rank $n$ of simple type. If $ 0 = E_0\subset E_1 \subset \cdots \subset E_m = E$ is its  HN-filtration
then
\begin{equation}\label{coro3}
\dim End(E)= 1+  h^0(F_{m-1}^*\otimes E_{m-1}).
\end{equation}
Moreover, $h^0(F_{m-1}^*\otimes E_{m-1})\leq \sum _0 ^{m-2} h^0(F_{m-1}^*\otimes F_{i}).$
\end{corollary}

\begin{proof} We give only the main ideas of the proof.  We can proceed analogously to the proof of Corollary \ref{cor2}. The equality  $\dim End(E)= 1+  h^0(F_{m-1}^*\otimes E_{m-1})$  follows from Lemma \ref{lemfin} and the cohomology of the exact sequences
$$0\to E_{m-1}\otimes E^*\to E\otimes E^* \to F_{m-1}\otimes E^* \to 0,$$
$$0\to F_{m-1}^*\otimes F_{m-1}\to E^*\otimes F_{m-1} \to E_{m-1}^*\otimes F_{m-1} \to 0,$$
$$0\to F_{m-1}^*\otimes E_{m-1}\to E^*\otimes E_{m-1} \to E_{m-1}^*\otimes E_{m-1} \to 0,$$
since $E$ is of simple type, and hence HN-indecomposable.

The basic idea of the proof of the bound for $h^0(F_{m-1}^*\otimes E_{m-1})$  is to take the cohomology of the sequences
$$0\to E_{i-1}\otimes F_{m-1}^*\to E_i\otimes F_{m-1}^* \to F_{i-1}\otimes F_{m-1}^* \to 0,$$
and proceed by induction.
The details are left to the reader.
\end{proof}

\section{Moduli spaces} \label{moduli}

In this section Proposition \ref{propfm} is given in a more general setting. Our aim is to consider families of semistable bundles. Recall that if the vector bundles $E_1$ and $F_1$ have automorphisms the $\mathbb{C}^*$-action does not identify all the isomorphic classes of the vector bundles in $H^1(X,E_1\otimes F_1^*)$.  The simple vector bundles over X with fixed rank and degree possess a
coarse moduli space (see \cite[Corollary 6.5]{KO}), but there is no universal family. It is possibly that it is non-separated and by the work of M. Artin is an algebraic space. Therefore, to construct a moduli scheme of indecomposable vector bundles of HN-length $2$ we will consider those of coprime type $\sigma=(\mu _1, \mu _2)$. Note that if  $\mu _1$ is given by $\mu _1=\frac{d_1}{n_1}$ then $\mu _2=\frac{d-d_1}{n-n_1}:=\frac{d_2}{n_2}.$  From now on we tacitly assume that $\mu _1=\frac{d_1}{n_1} >\frac{d_2}{n_2}=\mu _2 $ and denote by $U_{{\mu _1}}(n,d)$ the set of such bundles. Let
$$U_{{\mu _1}}(n,d,k):=\{E\in U_{{\mu _1}}(n,d): \dim End(E)=1+k\}.$$

The moduli space $M(n_i,d_i)$ of semistable vector bundles of rank $n_i$ and degree $d_i$ for $1=1,2$ with $gcd(n_i,d_i)=1$
carries a universal bundle
$\mathcal{U}_i \to X\times {M}(n_i,d_i)$. In general, $\mathcal{U}_i$ is determined up to tensoring by a line bundle lifted from ${M}(n_i,d_i)$. In this paper it will be fixed, unless
otherwise stated.
Denote by $p_{ij}$ the projection of $X\times {M}(n_1,d_1) \times {M}(n_2,d_2)$ in the $ij$-factors. We will denote by ${\mathcal{R}}_{{\mu _1}}$ the $1$st- direct image sheaf
$${\mathcal{R}}_{{\mu _1}}:=\mathcal{R}^1_{{p_{23}}}(p_{12}^*{\mathcal{U}_{{1}}}\otimes p_{13}^*{\mathcal{U}^*_{{2}}})$$ over ${M}(n_1,d_1) \times {M}(n_2,d_2)$.

With the notation $M_i:=M(n_i,d_i)$, the following diagram summarise the notation.

\begin{equation}\label{diag1}
\xymatrix@1{ & p_{12}^*\mathcal{U}_{{1}}\otimes p_{13}^*\mathcal{U}^*_{{2}}\ar[d] &
\mathcal{R}_{{\mu _1}}:=\mathcal{R}^1_{{p_{23}}}(p_{12}^*\mathcal{U}_{{1}}\otimes
p_{13}^*\mathcal{U}^*_{{2}})\ar[d] \\
\mathcal{U}_{{1}}\ar[d] & X\times M_1 \times M_2\ar[dl]^{p_{12}}\ar[r]^{p_{23}}\ar[dr]^{p_{13}} &M_1 \times M_2\\
X\times M_1& & X\times M_2.\\}
\end{equation}

\bigskip

The following theorem follows from \cite{grot}, \cite{lange}, Proposition \ref{indecom} and Lemma \ref{lema1}.

\begin{theorem}\label{teo1} If $\mu _1-\frac{d_2}{n_2}\geq 2g-1$, $U_{{\mu _1}}(n,d)=\emptyset$ and if  $\mu _1-\frac{d_2}{n_2}< 2g-1$, $U_{{\mu _1}}(n,d)$  has a projective scheme structure
that makes it an moduli scheme. Moreover, there is a natural isomorphism between $U_{{\mu _1}}(n,d)$ and $Proj(\mathcal{R}_{{\mu _1}})$.
\end{theorem}

\begin{proof}
Let $E\in U_{{\mu _1}}(n,d)$. Note that being of coprime type implies that $E$ is an extension of two stable bundles, that is
$$\rho _1: 0\to E_1 \to E \to F_1 \to 0.$$
Since  ${\it Ext}^2(p_{13}^*{\mathcal{U}_{{2}}}, p_{12}^*{\mathcal{U}_{{1}}})=0$, it follows that $\mathcal{R}^1_{{p_{23}}}(p_{12}^*{\mathcal{U}_{{1}}}\otimes p_{13}^*{\mathcal{U}^*_{{2}}})_{{|_{(E_1,F_1)}}} \to
H^1(X, E_1\otimes F_1^*)$ is an isomorphism.  From \cite{grot}, \cite{lange}, the coherent sheaf ${\mathcal{R}}_{{\mu _1}}:=\mathcal{R}^1_{{p_{23}}}(p_{12}^*{\mathcal{U}_{{1}}}\otimes p_{13}^*{\mathcal{U}^*_{{2}}})$ parameterizes the classes of extensions of two stable bundles $(E_1,F_1)\in M_1\times M_2$.   From Lemma \ref{decomposable},$(1)$ $E$ is indecomposable if and only if $\rho _2\ne 0$. Therefore, the theorem follows from Remark \ref{rem1}.
\end{proof}

\begin{remark}\begin{em} Note that the coherent sheaf ${\mathcal{R}}_{{\mu _1}}$ could be defined even if $\gcd(n_0,d_0)\ne1$. Indeed, if $\gcd(n_i,d_i)\ne1$, there exists an \'etale covering $\widetilde{M_i}$ of $M_i$ such that a universal bundle $\widetilde{\mathcal{U}_i}$ exists  on  $X\times \widetilde{M_i}$ (see \cite[Proposition 2.4]{nr}).
Thus, the coherent sheaf $\widetilde{{\mathcal{R}}_{{\mu _1}}}:=\mathcal{R}^1_{{p_{23}}}(p_{12}^*{\widetilde{\mathcal{U}_{{1}}}}\otimes p_{13}^*{\widetilde{\mathcal{U}^*_{{2}}}})$ will parameterize extensions of two semistable bundles $(E_1,F_1)\in \widetilde{M_1}\times \widetilde{M_2}$. However, in this case we can not apply Lemma \ref{decomposable},$(1)$ (see Remark \ref{rem1}) since semistable bundles can be non-simple.
\end{em}
\end{remark}

In general, the coherent sheaf ${\mathcal{R}}_{{\mu _1}}=\mathcal{R}^1_{{p_{23}}}(p_{12}^*{\mathcal{U}_{{1}}}\otimes p_{13}^*{\mathcal{U}^*_{{2}}})$ is not locally free. We will give a flattening stratification of $M_1 \times M_2$ for ${\mathcal{R}}_{{\mu _1}}$. To give the schematic semi-continuity stratification of $M_1 \times M_2$ we will use the algebra of endomorphisms and the twisted Brill-Noether theory. For a recent account of the theory we refer the reader to \cite{hitching}. We follow \cite{arb} in the construction of the determinantal varieties that we need.

Let $D$ be an effective divisor on $X$ of degree $d_0>>0$ such that for any $(E_1,F_1)\in M_1 \times M_2$, $H^1(X,E_1\otimes F_1^*\otimes \mathcal{O}(D))=0.$ Let

$$0\to H^0(X,E_1\otimes F_1^*) \to H^0(X,E_1\otimes F_1^*\otimes \mathcal{O}(D)) \stackrel{\phi}{\to} H^0(X,(E_1\otimes F_1^*)_{{|_D}}) \to H^1(X,E_1\otimes F_1^*)\to 0$$
be the cohomology sequence of
$$0\to E_1\otimes F_1^* \to E_1\otimes F_1^*\otimes \mathcal{O}(D) \to (E_1\otimes F_1^*)_{{|_D}} \to 0.$$
Thus, $h^0(E_1\otimes F_1^*) \geq k$ if and only if $rk(\phi) < h^0(E_1\otimes F_1^*\otimes \mathcal{O}(D))-k$.

We want to vary $(E_1, F_1)$ in $ M_1\times M_2.$
Let $\Gamma :=D\times M_1 \times M_2$ be the product divisor in $X\times M_1 \times M_2$. We have the following sequence
$$ 0\to p_{12}^*{\mathcal{U}_{{1}}}\otimes p_{13}^*{\mathcal{U}^*_{{2}}}\to p_{12}^*{\mathcal{U}_{{1}}}\otimes p_{13}^*{\mathcal{U}^*_{{2}}}\otimes \mathcal{O}(\Gamma)\to \mathcal{N} \to 0,   $$
over $X\times M_1 \times M_2$, where $\mathcal{N}$ is the quotient $p_{12}^*{\mathcal{U}_{{1}}}\otimes p_{13}^*{\mathcal{U}^*_{{2}}}\otimes \mathcal{O}(\Gamma)/ p_{12}^*{\mathcal{U}_{{1}}}\otimes p_{13}^*{\mathcal{U}^*_{{2}}}.$

 Since $H^1(E_1\otimes F_1^*\otimes \mathcal{O}(D))=0$, the direct image induces the complex $\phi :\mathcal{K}_0\to \mathcal{K}_1$ of locally free sheaves over  $M_1 \times M_2$ where $\mathcal{K}^0:=\mathcal{R}^0_{{p_{23}}}(p_{12}^*\mathcal{U}_{{1}}\otimes p_{13}^*\mathcal{U}^*_{{2}}\otimes \mathcal{O}(\Gamma))$ is of rank $d(E_1\otimes F_1^*\otimes \mathcal{O}(D) ) +n_1n_2(1-g) $ and $\mathcal{K}^1:=\mathcal{R}^0_{{p_{23}}}(\mathcal{N})$.  We follow the notation of \cite{hitching} and denote by $B^{k}(\mathcal{U}_1,\mathcal{U}^*_2)$ the $k$th-determinantal variety of the complex $\phi :\mathcal{K}_0\to \mathcal{K}_1$.  That is,
 \begin{equation}\label{support}
 Supp(B^{k}(\mathcal{U}_1,\mathcal{U}^*_2)):=\{(E_1,F_1)\in M_1 \times M_2: h^0( E_1\otimes F_1^*)\geq k\}.
 \end{equation}

\begin{remark}\begin{em}\label{remnotation}When no confusion can arise and to simplify notation, from now on we use the following notation:
\begin{itemize}
\item $n=n_1+n_2$,
\item $d=d_1+d_2$,
\item $n_0=n_1n_2$,
\item $d_0=n_2d_1-n_1d_2$,
\item $h^0=k$ and
\item $h^1=k-d_0 +n_0(g-1)$.
\end{itemize}
\end{em}\end{remark}

Hence, as $k$th-determinantal variety, the {\it expected dimension} of $B^{k}(\mathcal{U}_1,\mathcal{U}^*_2)$ is the number
 $$
 \begin{array}{cll}
 \rho (g,n_1,d_1,n_2,d_2,k)&:=&\dim (M_1 \times M_2)-h^0\cdot h^1\\
 &=&(n_1^2+n_2^2)(g-1) +2-k(k-d_0 +n_0(g-1)).\\
\end{array}
$$
and $B^{k+1}(\mathcal{U}_1,\mathcal{U}^*_2)\subseteq  Sing B^{k}(\mathcal{U}_1,\mathcal{U}^*_2).$ Moreover, the filtration
$$M_1 \times M_2\supset B^{1}(\mathcal{U}_1,\mathcal{U}^*_2)\supset  \cdots  \cdots \supset B^{k}(\mathcal{U}_1,\mathcal{U}^*_2)\supset B^{k+1}(\mathcal{U}_1,\mathcal{U}^*_2) \supset \cdots \cdots $$
define a stratification by closed subsets.
We will denote by $\mathcal{Y}_k$ the stratum

$$\mathcal{Y}_k:=B^{k}(\mathcal{U}_1,\mathcal{U}^*_2)-B^{k+1}(\mathcal{U}_1,\mathcal{U}^*_2).$$

That is, $$\mathcal{Y}_k=\{(E_1,F_1)\in M_1 \times M_2: h^0(E_1\otimes F_1^*)=k\}. $$ For any $(E_1,F_1)\in \mathcal{Y}_k$,
$$\dim H^1(X,E_1\otimes F_1^*)=k-d(E_1\otimes F_1^*)+n_1n_2(g-1)=k-d_0+n_0(g-1).$$ Therefore, {\it the restriction of} ${\mathcal{R}}_{{\mu _1}}$ to $\mathcal{Y}_k$, denoted by ${\mathcal{R}}_{{\mu _1}}(\mathcal{Y}_k)\to \mathcal{Y}_k$, is locally free of rank $k-d_0 +n_0(g-1)$. Let
$$\mathbb{P}({\mathcal{R}}_{{\mu _1}}(\mathcal{Y}_k))\to \mathcal{Y}_k $$
be the projective bundle associated to ${\mathcal{R}}_{{\mu _1}}(\mathcal{Y}_k)\to \mathcal{Y}_k$.

{\it The expected dimension} of $\mathbb{P}({\mathcal{R}}_{{\mu _1}}(\mathcal{Y}_k))$ is
$$
\begin{array}{cll}
\beta(g,n_1,d_1,n_2,d_2,k)&:=&\dim (M_1 \times M_2)-h^0h^1 +h^1-1\\
&=&\dim (M_1 \times M_2)-h^1(h^0 -1) -1\\
&=&(n_1^2+n_2^2)(g-1) +1 -(k-1)(k-d_0 +n_0(g-1)).
\end{array}
$$

\begin{remark}\begin{em}\label{remy0} $\mathcal{Y}_0$ is the locus where $ h^0(E_1\otimes F_1^*)=0$, and hence, the indecomposable bundles $E\in H^1(X,E_1\otimes F_1^*)$  are simple (see Corollary \ref{cor2}).
Under the assumption that  $M_1 \times M_2\ne B^{1}(\mathcal{U}_1,\mathcal{U}^*_2)$,
$$\mathcal{Y}_0=M_1 \times M_2- B^{1}(\mathcal{U}_1,\mathcal{U}^*_2)$$
is an open set and, in consequence,
$$\dim \mathbb{P}({\mathcal{R}}_{{\mu _1}}(\mathcal{Y}_0))=\dim M_1+\dim M_2 +h^1-1.$$
\end{em}\end{remark}

\begin{remark}\begin{em} \begin{enumerate}
\item If $p_2:B^{k}(\mathcal{U}_1,\mathcal{U}^*_2)\subset M_1\times M_2 \to  M_2$ is the projection then for $F_1\in p_2(B^{k}(\mathcal{U}_1,\mathcal{U}^*_2)) $ the inverse image $p_2^{-1}(F_1)$ is the locus
$$B^{k}(\mathcal{U}_1,F^*_1):=\{E_1\in M_1:h^0(E_1\otimes F_1^*)\geq k \}.$$
\item In particular, if $M_2=Pic^0(X)$ and $F_1=\mathcal{O}_X$ then $B^{k}(\mathcal{U}_1,\mathcal{O}_X)$ is the Brill-Noether locus $B(n_1,d_1,k)$. In this case we denote $\mathcal{Y}_k$ as $Y_k$. That is,
$$Y_k=B(n_1,d_1,k)-B(n_1,d_1,k+1).$$
\item In general, the product of two stable bundles is semistable. However, from  \cite[Lemma 3.5]{bbnpic} if on of the bundles is general then $E\otimes F^* $ is stable, or if $(n_0,d_0)=1$. Thus, in this case if
 $ B^{k}(\mathcal{U}_1,\mathcal{U}^*_2)\ne \emptyset$ then $B(n_0,d_0,k)\ne \emptyset $.
\item  If $M_1 \times M_2\ne B^{k}(\mathcal{U}_1,\mathcal{U}^*_2)$ and $B^{k}(\mathcal{U}_1,\mathcal{U}^*_2)\ne \emptyset$ then $\mathcal{Y}_k\ne \emptyset$.
\end{enumerate}
\end{em}\end{remark}

Using the notation of Remark \ref{remnotation} we can now formulate one of our main results.
For any $0\leq k\leq \frac{d_0}{2}+n_0,$ let $U_{{\mu _1}}(n,d,k)$ be the set
$$U_{{\mu _1}}(n,d,k):=\{E\in U_{{\mu _1}}(n,d): \dim End(E)=1+k\}.$$

\begin{theorem}\label{teo2} $U_{{\mu _1}}(n,d,k)=\mathbb{P}({\mathcal{R}}_{{\mu _1}}(\mathcal{Y}_k))$ is coarse moduli space and $U_{{\mu _1}}(n,d,0)$ is a fine moduli space.
Moreover, if $\mathcal{Y}_k$ is irreducible and smooth of dimension $\rho$, then $U_{{\mu _1}}(n,d,k)$ is irreducible and smooth of dimension $\rho + h^1-1. $
\end{theorem}

\begin{proof}  From what has already
been proved and Corollary \ref{cor2} it follows that $$U_{{\mu _1}}(n,d,k)=\mathbb{P}({\mathcal{R}}_{{\mu _1}}({\mathcal{Y}_k})).$$

If $E\in U_{{\mu _1}}(n,d,0)$, E is simple and hence from Corollary \ref{cor2} $h^0(E_1\otimes F_1^*)=0$. Thus, $(\mathcal{R}^0_{{p_{23}}}(p_{12}^*{\mathcal{U}_{{1}}}\otimes p_{13}^*{\mathcal{U}^*_{{2}}})){{|_{{\mathcal{Y}_0}}}}=0$. Hence, from
\cite[Corollary 4.5]{lange}, $\mathbb{P}({\mathcal{R}}_{{\mu _1}}(\mathcal{Y}_0))$ parameterize a universal extension
$$0\to q^*(p_{12}^*(\mathcal{U}_1))\otimes p_2^*\mathbb{H}\to \mathcal{E}\to q^*p_{13}^*(\mathcal{U}_2)\to 0,$$
where $q:X\times \mathbb{P}({\mathcal{R}}_{{\mu _1}}(\mathcal{Y}_0)) \to X\times \mathcal{Y}_0$ is the induced map, $p_2:X\times \mathbb{P}({\mathcal{R}}_{{\mu _1}}(\mathcal{Y}_0))\to \mathbb{P}({\mathcal{R}}_{{\mu _1}}(\mathcal{Y}_0))$ is the projection  and $\mathbb{H}$ the hyperplane bundle over $\mathbb{P}({\mathcal{R}}_{{\mu _1}}(\mathcal{Y}_0))$.
The universal properties of the family $\mathcal{E}$ imply that the pair $(\mathbb{P}({\mathcal{R}}_{{\mu _1}}(\mathcal{Y}_0)),\mathcal{E}) $ is the fine moduli space for
$U_{{\mu _1}}(n,d,0)$.

If $\mathcal{Y}_k$ is irreducible and smooth of dimension $\rho$, a straightforward computation shows $U_{{\mu _1}}(n,d,k)$ is irreducible and smooth of dimension $\rho + h^1-1$,  since $\mathbb{P}({\mathcal{R}}_{{\mu _1}}(\mathcal{Y}_k)) $ is a projective bundle with fibre $\mathbb{P}(H^1(X, E_1\otimes F_1^*))$ at $(E_1,F_1)\in \mathcal{Y}_k$.
\end{proof}

\begin{corollary}\label{corprin0}
If $U_{{\mu _1}}(n,d,k)$ is non-empty, $B^{k}(\mathcal{U}_1,\mathcal{U}^*_2)$ is non-empty.
\end{corollary}

\begin{corollary}\label{corprin} If $\mathcal{Y}_k$ is irreducible and smooth then $H^i(\mathcal{U}_{{\mu _1}}(n,d,k), \mathbb{C})\cong H^i(\mathcal{Y}_k, \mathbb{C})$ for $i\geq 0$.
\end{corollary}

\begin{remark}\begin{em} The moduli space of simple bundles of type $\sigma=(\mu (E_1), \dots ,\mu( E_m/E_{m-1}))$, for $m>2$, are considered in \cite{buns}.
\end{em}\end{remark}

\begin{remark}\begin{em}  Theorem \ref{teo2} expresses the equivalence of the existence and topology of the $U_{{\mu _1}}(n,d,k)$ and that of twisted Brill-Noether loci. For some values of $(g, n,d,k)$, non emptiness, dimension, and irreducibility of $B^{k}(\mathcal{U}_1,\mathcal{U}^*_2)$, and of $\mathcal{Y}_k$, are known for general curve (see \cite{hitching}). Thus, as in the Brill-Noether and twisted Brill-Noether theory for vector bundles, it is possible that for special curves the moduli space $U_{{\mu _1}}(n,d,k)$ is  even reduced. Thus, the above results differ from the corresponding results for the moduli space of
stable bundles, where non-emptiness, dimension etc. are independent of the curve.
\end{em}\end{remark}

\section{Non-emptiness of $U_{{\mu _1}}(n,d,k)$ }\label{nonempty}

 In this section we will prove non emptiness of $U_{{\mu _1}}(n,d,k)$ for some values of $(g,n,d,k)$. First we consider the case of bundles of type $\sigma =(\mu _1, \mu _2)$ where $\mu _1=\frac{d-a}{n-1}$ and $\mu _2=a$ is an integer. That is, the quotient of the maximal destabilizing subbundle is a line bundle of degree $a$. In this case $\mathcal{Y}_k $ is a subset of the Brill-Noether locus $B(n-1,d-an,k)$, and will be denoted as $Y_k$.

The next theorems are applications the known results on $B(n-1,d-an,k)$ and Theorem \ref{teo2}.

\begin{theorem}\label{teop3} Let $n,d,a $ and $k$ be positive integers.  Assume $0<d-an<2(n-1)$  and $(n-1,d-an,k)\ne (n-1,n-1,n-1)$. Then for $\mu _1=\frac{d-a}{n-1}$, $U_{{\mu _1}}(n,d,k)$ is non-empty if and only if $k\leq n-1+\frac{d-n(a+1)+1}{g}$.
Moreover, if $U_{{\mu _1}}(n,d,k)$ is non-empty then it is irreducible and smooth of the expected dimension.
\end{theorem}

\begin{proof} The theorem follows from Theorem \ref{teo2}, since from \cite[Theorem A]{bgn} and  \cite[A-1 Th\'{e}or\`{e}me]{mer}, $Y_k\subset B(n-1,d-an,k)$ has the required properties.
\end{proof}

\begin{remark}\begin{em} Note that under the assumptions of Theorem \ref{teop3}, the positivity of the expected dimension $\rho (n-1,d-an,1,a,k)$ does not imply non emptiness of $U_{{\mu _1}}(n,d,k) .$
\end{em}\end{remark}

For $g=2$ we have a complete description.

\begin{corollary} Assume  $g=2$. If $0<d-an<2(n-1)$ and $0\leq k\leq \frac{d-an}{2}+\frac{n-1}{2}$ then  $U_{{\mu _1}}(n,d,k)$ is irreducible and smooth of the expected dimension. If $d-an\geq 2(n-1)$, $U_{{\mu _1}}(n,d,k)= \emptyset .$
\end{corollary}

We now rephrase Theorem \ref{bn},(5),(6) and (7) as follows.

\begin{theorem}\label{teop4}
 Let  $(g,n-1,d-na,k)$  be integers that satisfies the conditions given in  Theorem \ref{bn},(5),(6) and (7).
 For general curve,  $U_{{\mu _1}}(n,d,k)$ is non-empty and has an irreducible component of the expected dimension and if  $X$ is a Petri curve of genus $g\geq 3,  n\geq 5$ and $g\geq 2n-4$ then $U_{{\mu _1}}(n,d,n)$ is non-empty.
\end{theorem}

In the above case, to the best of our knowledge, it is know that there exists an irreducible component of the expected dimension, however, in general,  it is not known that $B(n,d,k)$ is irreducible.

Let us  now consider the general case. We will use the notation given in Remark \ref{remnotation}. Recall that $B^{k}(\mathcal{U}_1,\mathcal{U}^*_2) \subset M_1\times M_2$.
 Denote by $M_{12}$ the moduli space of stable bundles ${M}(n_0,d_0 )$.  From what has already been proved, we conclude that if $B(n_0,d_0,k)=\emptyset $ then $U_{{\mu _1}}(n,d,k)=\emptyset .$ However, if $B(n_0,d_0,k)\ne \emptyset $ does not imply that $U_{{\mu _1}}(n,d,n)\ne \emptyset .$
Let $\Phi :M_1\times M_2 \to M_{12}$ be the morphism  defined as $(E, F)\mapsto E\otimes F^*$. This gives
$$\Phi (B^{k}(\mathcal{U}_1,\mathcal{U}^*_2))  \subset B(n_0,d_0,k).$$ We want to describe $\Phi (B^{k}(\mathcal{U}_1,\mathcal{U}^*_2))$ in some cases.

\begin{remark}\begin{em}\label{remmulti} Let $G$ be a vector bundle of rank $m$ and degree $d_2$ generated by linear subspace $V\subset H^0(X,E)$ of dimension $k_2=m+n_2$. Let $D_{E,V}^*$ be the kernel of the evaluation map (or the  syzygy bundle). Let us introduce the temporary notation $F_1^*$ for $D_{E,V}^*$.  That is, $F_1^*$ has rank $rk (F_1)=n_2$, degree $d(F_1^*)=-d_2$ and fits into the following exact sequence
\begin{equation}\label{eqdualm}
 0\to F_1^*\to V\otimes \mathcal{O}\to G\to 0.
\end{equation}
Tensor (\ref{eqdualm}) with a vector bundle $E_1$ of rank $n_1$ and degree $d_1$ with $h^0(E_1)\geq k_1$. The injectivity of the multiplication map $\mu _{{V,E_1}}: V\otimes H^0(E_1)\to H^0(G\otimes E_1)$ is measure by $h^0(F_1^*\otimes E_1)$. Indeed, if $H^0(F_1^*\otimes E_1)=0$, $\mu _{{V,E_1}}$ is injective.   From the cohomology sequence
$$0\to H^0(F_1^*\otimes E_1)\to V\otimes H^0(E_1)\to H^0(G\otimes E_1)\to$$
we obtain the inequality
\begin{equation}\label{eqdesigk}
\dim V\cdot h^0(F)-  h^0(G\otimes E_1)\leq h^0(F_1^*\otimes E_1).
\end{equation}
Thus, if
\begin{equation}\label{eqdesigk1}
k\leq k_1k_2-h^0(G\otimes E_1)
\end{equation}
 then $k\leq h^0(F_1^*\otimes E_1).$
\end{em}\end{remark}

 With the above notation,  $n_0=rk(F_1^*\otimes E_1)=n_1n_2$ and  $d_0=d(F_1^*\otimes E_1)=d_1n_2-d_2n_1.$ Let $\tilde{n}:= rk(G\otimes E_1)=mn_1$ and $\tilde{d}:=d(G\otimes E_1)=d_2n_1+d_1m$.
Our interest is to apply Remark \ref{remmulti} to situations in which $E_1$ is general in $M(n_1,d_1)=M_1$,  $F_1\in M(n_2,d_2)=M_2 $ and $k\geq 0$, in particular when $\mu _{{V,E_1}}$ is not injective. In that case, $(E_1,F_1^*) \in B^k(U_1,U_2^*)$ and $E_1\otimes F_1^*\in B(n_0,d_0,k)$. The following theorems give existence of some Brill-Noether and twisted Brill-Noether loci, and therefore the non emptiness of $U_{{\mu_1}}({n},{d},k)$. The proofs are different of those of \cite{hitching} and some values, to our best knowledge, are not been include in \cite{hitching} nor in Theorem  \ref{bn} (see \cite{tbn}).

\begin{theorem}\label{teop05}
 Assume that $B(n_1,d_1,n_1+a)$ is non-empty with $a>0$. If $2n_1<d_1<a(g+1)$ and $d_2> 2gm$ then for any $0\leq k\leq (d_2+m(1-g))(n_1+a) -(d_2n_1+d_1m +mn_1(1-g))$,   $\mathcal{Y}_k\subset B^k(U_1,U_2^*)$ is non-empty. Moreover, if $\mu_1 = \frac{d_1}{n_1}$ then $U_{{\mu_1}}({n},{d},k)$ is non-empty, where $n=n_1+n_2$ and $d=d_1+d_2$.
\end{theorem}

\begin{proof} Let $G\in M(m,d_2)$ and $E_1\in B(n_1,d_1,n_1+a)$. We want to prove that $0\leq k_1k_2-h^0(G\otimes E_1)$, where $k_1=n_1+a$ and $k_2=d_2+m(1-g)$.

  If $d_2> 2gm$, $G$ is generated and from \cite[Theorem 1.2]{but} we have the exact sequence
\begin{equation}\label{eqdual01}
0\to F_1^*\to H^0(G)\otimes \mathcal{O}\to G\to 0
\end{equation}
with  $F_1\in M(n_2, d_2) $ where $n_2:= h^0(G)-m=d_2-mg.$
From Remark \ref{remmulti} and (\ref{eqdesigk1})
\begin{equation}\label{eqdesigk2}
  h^0(G)\cdot (n_1+a) \leq h^0(G)\cdot h^0(E_1)- h^0(G\otimes E_1)\leq h^0(E_1\otimes F_1^*).
\end{equation}

If $d_1<a(g+1)$ and $d_2> 2gm$ then  $\frac{d_1}{a}+g-1<2g<\frac{d_2}{m}$. Hence,
\begin{equation}\label{eqdesig}
d_1m+am(g-1)<ad_2.
\end{equation}
Now add in both sides of \ref{eqdesig} $d_2n_1+mn_1(g-1)$ to obtain
$$
d_2n_1+d_1m +(n_1+a)m(g-1)<d_2(n_1+a)+n_1m(g-1).
$$
Hence,
\begin{equation}\label{eqdesigg}
0<(n_1+a)(d_2+m(1-g))-(d_2n_1+d_1m + mn_1(1-g))=k_1k_2-h^0(G\otimes E_1).
\end{equation}
This gives $0\leq k\leq h^0(E_1\otimes F_1^*)$, for any $0\leq k\leq (d_2+m(1-g))(n_1+a) -(d_2n_1+d_1m +mn_1(1-g))$. It follows that $B^k(U_1,U_2^*)$ is non-empty and, in consequence,  $U_{{\mu_1}}(n,d,k)$ is non-empty.
\end{proof}

\begin{theorem}\label{teopetrif}  Let $X$ be a Petri curve of genus $g\geq 3$ and $(\mathcal{O}(D), V)$ a general generated linear system of degree $d_2\geq g+1$ and $\dim V =n_2+1$ with $n_2\leq 4 $ or if $n_2\geq 5$ then $g \geq  2n_2-4$. Assume that $B(n_1,d_1,t)$ is non-empty and $\frac{d_2}{n_2}<\frac{d_1}{n_1}$. For any $0\leq k \leq n_2t-n_1d_2$, $ \mathcal{Y}_k \subset B^k(U_1,U_2^*)$ is non-empty and if $n=n_2+n_1$, $d=d_2+d_1$ and $\mu _1=\frac{d_1}{n_1}$,  $U_{{\mu_1}}(n,d,k)$, is non-empty.
\end{theorem}

\begin{proof} Under the above conditions there exist an exact sequence
\begin{equation}\label{eqdual02}
0\to F_1^*\to V\otimes \mathcal{O}\to \mathcal{O}(D)\to 0
\end{equation}
such that $F_1\in M(n_2,d_2) $ (see \cite{b}, \cite{bbn1} and \cite{bbn2}).

Let $E_1\in B(n_1,d_1,n_1+a)$ such that $h^0(E_1)=t\geq n_1+a$. From Remark \ref{remmulti} and (\ref{eqdesigk1})
\begin{equation}\label{eqdesigk02}
   (n_2+1)\cdot t-  h^0(\mathcal{O}(D)\otimes E_1)\leq h^0(F_1^*\otimes E_1).
\end{equation}
From the assumtion $\frac{d_2}{n_2}<\frac{d_1}{n_1}$ and the exact sequence
\begin{equation}
0\to E_1\to \mathcal{O}(D)\otimes E_1\to (E_1)_D\to 0
\end{equation}
we deduce that
\begin{equation}\label{eqdesigk002}
 tn_2-d_2n_1 < (n_2+1)\cdot t-  h^0(\mathcal{O}(D)\otimes E_1)\leq h^0(F_1^*\otimes E_1).
\end{equation}
Therefore, for any $0\leq k\leq  tn_2-d_2n_1$,  $\mathcal{Y}_k \subset B^k(U_1,U_2^*)$ is non-empty and if $n=n_2+n_1$, $d=d_2+d_1$ and $\mu _1=\frac{d_1}{n_1}$,  $U_{{\mu_1}}(n,d,k)$, is non-empty, as claimed.
\end{proof}

\begin{remark}\begin{em} Under the hypothesis of Theorems \ref{teop05} and \ref{teopetrif} we can choose the $d_2,n_2,n_1$ and $d_1$ such that $(n_0,d_0)=1$, and hence $B(n_0,d_0,k) $ is non-empty.
\end{em}\end{remark}

We conclude now with a description of a smooth point of $U_{{\mu _1}}(n,d,k)$. Recall that $U_{{\mu _1}}(n,d,k)$ is a projective bundle over $\mathcal{Y}_k$.
Let $0\subset E_1\subset E$ be the HN-filtration of $E\in U_{{\mu _1}}(n,d,k)$. Write $E/E_1=F_1$.  
Let us describe the tangent bundle of $B^{k}(\mathcal{U}_1,\mathcal{U}^*_2)$ at a point $(E_1,F_1^*)$. We abbreviate $B^{k}(\mathcal{U}_1,\mathcal{U}^*_2)$ to $\mathcal{B}_{12}$.  Since $\mathcal{B}_{12}\subset M_1\times M_2$, for any $z:=(E_1,F_1)\in \mathcal{B}_{12}$ we have the sequence
\begin{equation}\label{tan1}
0\to T_{{z}} \mathcal{B}_{12} \to T_{{E_1}}M_1\oplus T_{{F_1}}M_2\to N\to 0
\end{equation}
where $N$ is the normal bundle. Moreover, from the morphism $\Phi :M_1\times M_2 \to M_{12}$ we get
\begin{equation}\label{tan2}
 T_{{E_1}}M_1\oplus T_{{F_1}}M_2\stackrel{d\Phi}{\to } T_{{z}}M_{1,2}
\end{equation}
That is, 
$$H^1(X,End(E_1))\oplus H^1(X,End(F_1))\stackrel{d\Phi}{\to } H^1(X,End(E_1\otimes F_1^*)). $$
Moreover, from the dual of the multiplication map
$$ H^0(X,E_1\otimes F_1^*)\otimes H^0(X,(E_1\otimes F_1^*)^*\otimes K)\to H^0(X,End (E_1\otimes F_1^*)\otimes K)$$
we have the morphism
$$ H^1(X,End(E_1\otimes F_1^*))\stackrel{\beta}{\to } H^0(X,E_1\otimes F_1^*)^*\otimes H^1(X,E_1\otimes F_1^*). $$
Thus,
$$
\begin{array}{ccc}
H^1(X,End(E_1))\oplus  H^1(X,End(F_1))&\stackrel{d\Phi}{\to } & H^1(X,End(E_1\otimes F_1^*))\\
&\eta\searrow & \beta\downarrow \\
&&H^0(X,E_1\otimes F_1^*)^*\otimes H^1(X,E_1\otimes F_1^*)
\end{array}
$$
where $\eta=\beta \circ d\Phi $ and  the image of $\eta $ is precisely the normal bundle. Hence from \cite[Proposition 3.5]{hitching} and Theorem \ref{teo2} we conclude

\begin{theorem}\label{teop8} $U_{{\mu _1}}(n,d,k)$ is smooth at $E$ and of expected dimension
$$\beta (g,n_1,d_1,n_2,d_2,k)\   \mbox{if and only if} \  \eta \   \mbox{is surjective}. \hspace{6cm}$$
\end{theorem}


\end{document}